\documentclass[10pt,letter]{amsart}
\usepackage[utf8]{inputenc}
\usepackage[T1]{fontenc}
\usepackage{amsmath}
\usepackage{amssymb}
\usepackage{graphicx}
\usepackage{hyperref}
\usepackage{cleveref}
\usepackage{tikz-cd}
\usepackage[margin=2cm]{geometry}
\usepackage{adjustbox}

\theoremstyle{definition}
\newtheorem{thm}{Theorem}

\newtheorem{lemma}[thm]{Lemma}
\newtheorem{proposition}[thm]{Proposition}

\newtheorem{ex}[thm]{Example}

\newtheorem*{thm*}{Theorem}

\newcommand{\ZZ}{{\mathbb Z}}

\title[Obstructions to almost complex structures following Massey]{Obstructions to almost complex structures following Massey}

\begin{document}
\author{Michael Albanese}
\address{University of Adelaide, School of Computer and Mathematical Sciences, Level 1, Ingkarni Wardli Building, North Terrace Campus, 5005 Australia}
\email{michael.albanese@adelaide.edu.au}
\author{Aleksandar Milivojevi\'c}
\address{University of Waterloo, Faculty of Mathematics, 200 University Ave W, Waterloo, N2L 3G1, Ontario, Canada}
\email{amilivoj@uwaterloo.ca}

\subjclass[2020]{32Q60, 55S35, 53C15}
\keywords{Almost complex structures, obstruction theory}
 
 \begin{abstract} We provide proofs of two theorems stated by Massey in 1961, concerning the obstructions to finding complex structures on real vector bundles. In addition, we determine the second obstruction to a complex structure on a rank six orientable real vector bundle. The obstructions are fractional parts of integral Stiefel--Whitney classes, and a fourth of an appropriate combination of Pontryagin, Chern, and Euler classes. \end{abstract}

\maketitle

\section{Introduction}


In a 1961 paper, Massey \cite{M61} studied the question: When does a real vector bundle admit the structure of a complex vector bundle? In particular, when does a manifold admit an almost complex structure?


This problem has a long and rich history. As a starting point, one observes that on a complex vector bundle, certain relations must hold on the level of characteristic classes. First of all, since the Chern classes $c_i$ reduce modulo two to the even-indexed Stiefel--Whitney classes $w_{2i}$, one sees that the odd--indexed integral Stiefel--Whitney classes $W_{2i+1}$ must vanish, being the integral Bockstein of $w_{2i}$, and in particular $w_{2i+1}$ must vanish. Furthermore, the Pontryagin classes are related to the Chern classes by the formula $$ 1-p_1+p_2- \cdots = (1-c_1 + c_2 - \cdots )(1+c_1+c_2+ \cdots),$$ yielding $$(-1)^k p_k = \sum_{i+j = 2k} (-1)^i c_ic_j.$$ In addition, the top degree Chern class $c_n$ on a complex rank $n$ vector bundle coincides with its Euler class. 

With only this in hand, one can already show that spheres $S^{4k}$ in dimensions divisible by four admit no almost complex structure: being stably parallelizable, the Pontryagin classes vanish and hence so do the Chern classes, which would imply, in particular, that the Euler characteristic vanishes. To additionally rule out almost complex structures on spheres of dimension $4k+2$, for $k \geq 2$, Borel--Serre employed the cup product and Steenrod algebra structure on the cohomology of the classifying space $BU$ \cite[IV.15]{BS53}. Alternatively, one can reach the same conclusion using the integrality of the Chern character from $K$-theory to rational cohomology for spheres, due to Bott periodicity. Further $K$-theoretic computation shows that a product of even-dimensional spheres is almost complex exactly when it is a product of copies of $S^2, S^6, S^2 \times S^4$ \cite[Section 3]{S65}. Computations involving the Chern character were also used by Massey \cite{M62} to prove that quaternionic projective spaces $\mathbb{HP}^n$, for $n \geq 2$, are not almost complex. In low dimensions, we have a complete understanding of when a given orientable manifold admits an almost complex structure. For example, in dimension 6, the sole condition is that $W_3 = 0$ \cite[section 6]{E52}. The criterion in dimension 4 is a theorem of Wu \cite[Théorème IV.10]{W52}.


In many cases, what is preventing the existence of an almost complex structure is the inability to select classes $c_i$ which would be the Chern classes of the putative almost complex structure and hence satisfy the above relations. To systematically investigate this observation, one uses obstruction theory, whereby one identifies sufficient, not just necessary, conditions to finding an almost complex structure. By the time of Massey's paper \cite{M61}, it was understood how to relate the obstructions arising in the problem of lifting through the map $BU(n) \to BSO(2n)$ to the cohomology classes classifying the stages of its Moore--Postnikov system. This approach was developed by Hermann \cite{He59, He60}, whose technique Massey employed \cite[Section 3]{M61}. Finding a section of the $SO(2n)/U(n)$ bundle over a certain skeleton of the CW complex is tantamount to lifting the classifying map of the tangent bundle to a certain stage of the Moore--Postnikov system of the map $BU(n) \to BSO(2n)$. 

Massey thus observed that even though the existence of classes $c_i$, reducing mod 2 to $w_{2i}$, such that $1-p_1+p_2- \cdots = (1-c_1 + c_2 - \cdots )(1+c_1+c_2+ \cdots)$, is a necessary condition for having an almost complex structure, the actual obstructions in the relevant degrees are generally classes such that the vanishing of some non-trivial multiple of them corresponds to the identities $W_{2i+1} = 0$ and $(-1)^k p_k = \sum_{i+j = 2k} (-1)^i c_ic_j$. More precisely:

\begin{thm*}\cite[Theorems I and II]{M61} Let $\xi \to X$ be an orientable real vector bundle of rank $2n$, whose structure group we assume has been reduced to $SO(2n)$, over a CW complex $X$. Consider the associated $SO(2n)/U(n)$ bundle $\theta$, sections of which correspond to reductions of the structure group of $\xi$ to $U(n)$. If $s$ is a section of $\theta$ over the $q$--skeleton of $X$, then the obstruction to extending over the $(q+1)$--skeleton is a cohomology class $\mathfrak{o}_{q+1} \in H^{q+1}(X; \pi_{q} SO(2n)/U(n))$. We have the following: \begin{enumerate} \item If $q = 4k+2$ and $q < 2n-1$, then $$W_{4k+3}(\xi) = \begin{cases}
(2k)!\, \mathfrak{o}_{4k+3} \ \ \ \ &k \, \mathrm{even},\\
\tfrac{1}{2}(2k)!\, \mathfrak{o}_{4k+3} \ \ \ \ &\, k \, \mathrm{odd}.
\end{cases}$$
\item Assume $n = 2k$ and $q = 2n-1 = 4k-1$. The obstruction $\mathfrak{o}_{2n} = \mathfrak{o}_{4k}$ is an integral class if $k$ is odd, while if $k$ is even, it is a pair consisting of an integral class and a mod 2 class. If $k$ is odd, then $$\sum_{i+j = 2k} (-1)^i c_ic_j - (-1)^k p_k = 4 \mathfrak{o}_{2n},$$ while for $k$ even, this same formula holds with $\mathfrak{o}_{2n}$ replaced by its integral component. In the formula, $c_0, \ldots, c_{n-1}$ are the Chern classes of the induced $U(n)$ bundle over the $(2n-1)$--skeleton of $X$, while $c_n$ denotes the Euler class of $\xi$. \end{enumerate} \end{thm*}

In full generality, this does not address all the obstructions to finding an almost complex structure, and it is unclear whether the remaining obstructions are easily expressed in terms of characteristic classes.

Given that all the odd-indexed integral Stiefel--Whitney classes vanish on an almost complex manifold, one might wonder whether the classes $W_{4k+1}$ appear in certain obstructions. However, as we will see in \Cref{integralSW}, their vanishing is ensured by the vanishing of the classes $W_{4k+3}$. 

As an example of statement (1) above, the obstruction to extending an almost complex structure on the ten--skeleton of a twelve--manifold over its eleven--skeleton is not $W_{11}$, but rather a class $\mathfrak{o}_{11}$ satisfying $W_{11} = 24\mathfrak{o}_{11}$. Likewise, as an example of statement (2), four times the obstruction to extending an almost complex structure on the three--skeleton of a four--manifold over the four--skeleton (i.e., the whole manifold) is $p_1 - c_1^2 + 2e$, where $e$ is the Euler class. Of course, on an orientable four--manifold, this class vanishes if and only if $p_1 - c_1^2 + 2e$ vanishes, which is part of the aforementioned theorem of Wu, but on a general CW complex this need not be the case. 

Though Massey included remarks in \cite{M61} on the ingredients of the proofs of the above results, the proofs are not provided, and we take this opportunity to do so. Massey also addressed a particular obstruction not covered by the above, namely the degree eight obstruction to extending an almost complex structure on the seven--skeleton over the eight--skeleton, in the case of $2n \geq 10$ \cite[Theorem III]{M61}. This mod 2 obstruction to a stable almost complex structure, in the case of closed manifolds, was later studied in detail by Thomas \cite[Theorem 1.2]{T67} and Heaps \cite[\S 1]{H60}, and we will touch upon it only briefly in the last section.

It is wholly possible that some of the arguments we write out here for proving \cite[Theorems I and II]{M61} were not the ones Massey had in mind, particularly since we do not make use of any of the general statements (a)--(f) in \cite[Section 3]{M61} listed as ingredients of the proofs of \cite[Theorems I--III]{M61}, though perhaps these were mostly intended for the proof of \cite[Theorem III]{M61}. Our method, which only employs technology known at the time of \cite{M61}, is inspired by an argument of Kervaire used in \cite[Lemma 1.1]{K59}, which is mentioned in \cite[Remark 3]{M61}. 

\subsection*{Acknowledgements} We would like to thank Mark Grant for providing helpful references, together with the anonymous referee for comments that improved the exposition of the paper.

\section{Moore--Postnikov systems}

We begin with some general topological preliminaries we will make use of heavily, namely Postnikov systems of spaces, and their relative versions which we refer to as Moore--Postnikov systems.

\subsection{The Postnikov system of a space}\label{ordinarycase} For the content of this subsection, we refer the reader to e.g. \cite[Chapter IX]{Wh}. Given a simply connected CW complex $F$, there is a sequence of spaces $F_n$, $n\geq 2$, maps $F \xrightarrow{g_n} F_n$, and maps $F_n \xrightarrow{f_n} F_{n-1}$, such that:

\begin{itemize} \item $\pi_{>n} F_n = 0$.
\item The map $F \xrightarrow{g_{n}} F_n$ induces an isomorphism on $\pi_{\leq n}$.
\item The map $F_n \xrightarrow{f_n} F_{n-1}$ is a principal fibration with fiber $K(\pi_n F, n)$, classified by a map $F_{n-1} \to K(\pi_n F, n+1)$. The corresponding element in $H^{n+1}(F_{n-1}; \pi_n)$ is called the $k$--invariant in degree $n+1$.
\item $f_n \circ g_n$ is homotopic to $g_{n-1}$.
\end{itemize}

$$\begin{tikzcd}
                          & F \arrow[d, "g_n"]           \\
                          & F_n \arrow[d, "f_n"] \\ &\vdots \arrow[d, "f_5"]      \\
{K(\pi_4 F, 4)} \arrow[r] & F_4 \arrow[d, "f_4"]         \\
{K(\pi_3 F, 3)} \arrow[r] & F_3 \arrow[d, "f_3"]         \\
                          & {F_2 = K(\pi_2 F, 2)}
\end{tikzcd}$$

This data is referred to as a \emph{Postnikov system} for the space $F$. We refer to $F \xrightarrow{g_n} F_n$ as the $n^\mathrm{th}$ Postnikov map of $F$. Denoting the homotopy fiber of $F \xrightarrow{g_n} F_n$ by $F_n'$, we have $\pi_{\leq n} F_n' = 0$ and the map $F_n' \to F$ induces an isomorphism on $\pi_{>n}$. A map $X \to F_{n-1}$ lifts through $f_n$ if and only if the pullback of the $k$-invariant in $H^{n+1}(F_{n-1}; \pi_n)$ to $H^{n+1}(X; \pi_n)$ is the zero class.

\subsection{The Moore--Postnikov system of a map}\label{moorepostnikov} For the content of this subsection, we refer the reader to \cite{He60}, in particular Section 7, and to \cite[Chapter II.8]{M}, \cite[Chapter IX.6]{Wh}, \cite[Chapter 5.3]{Ba}. 

Consider a fibration consisting of simply connected CW complexes $F \to E \xrightarrow{p} B$. Define $E_1 = 
B$. There are spaces $E_n$, $n\geq 2$, and fibrations $F_n \to E_n \xrightarrow{p_n} B$, $F_n' \to E \xrightarrow{h_n} E_n$, and $K(\pi_n F, n) \to E_n \xrightarrow{e_n} E_{n-1}$ such that

\begin{itemize} \item the diagram

\begin{equation}\label{square} \begin{tikzcd}
E \arrow[d, "p"] \arrow[r, "h_n"] & E_n \arrow[d, "p_n"] \\
B \arrow[r, "id"]  & B                   
\end{tikzcd}\end{equation} commutes (up to homotopy), and the induced map of fibers $F \to F_n$ is an $n^\mathrm{th}$ Postnikov map for $F$ as described in \Cref{ordinarycase}. In particular, the Moore--Postnikov system of a constant map to a point induces a Postnikov system for its domain.

\item $E_n \xrightarrow{e_n} E_{n-1}$ is a principal $K(\pi_n F, n)$ fibration such that the diagram

$$\begin{tikzcd}
F_n \arrow[d] \arrow[r, "f_n"]        & F_{n-1} \arrow[d]            \\
E_n \arrow[d, "p_n"] \arrow[r, "e_n"] & E_{n-1} \arrow[d, "p_{n-1}"] \\
B \arrow[r, "id"]                     & B                           
\end{tikzcd}$$ is homotopy commutative, where $f_n$ is the induced map of homotopy fibers. Note that if $\pi_n F = 0$, then $E_n \xrightarrow{e_n} E_{n-1}$ is a homotopy equivalence. 
\end{itemize}

We can complete diagram \eqref{square} to a three--by--three diagram with the rows and columns being fibrations, see e.g. \cite[Section 3.2]{N10}, and from the long exact sequences in homotopy groups we see that the induced map from the homotopy fiber of $g_n$ (as in \Cref{ordinarycase}) to the homotopy fiber $F_n'$ of $h_n$ is a homotopy equivalence, justifying the notation $F_n'$ for the homotopy fiber of $h_n$. 

$$\begin{tikzcd}
F_n' \arrow[r] \arrow[d, "\sim"] & F \arrow[d] \arrow[r, "g_n"]      & F_n \arrow[d]        \\
F_n' \arrow[r] \arrow[d]                   & E \arrow[d, "p"] \arrow[r, "h_n"] & E_n \arrow[d, "p_n"] \\
                                \ast \arrow[r] & B \arrow[r, "id"]                 & B                   
\end{tikzcd}$$

Recall, the homotopy groups of $F_n'$ are trivial in degrees $\leq n$, while they are isomorphic to those of $F$ in degrees $>n$. In particular, the map $E \xrightarrow{h_n} E_n$ induces an isomorphism on $\pi_{\leq n}$ and a surjection on $\pi_{n+1}$, so $h_n^*$ is an isomorphism on cohomology in degrees $\leq n$ and injective in degree $n+1$. 

The data of the fibrations $F_n \to E_n \xrightarrow{p_n} B$, $F_n' \to E \xrightarrow{h_n} E_n$, and $K(\pi_n F, n) \to E_n \xrightarrow{e_n} E_{n-1}$ satisfying the above conditions is called a \emph{Moore--Postnikov system} for $F \to E \xrightarrow{p} B$. The principal $K(\pi_n F, n)$ fibration $E_n \xrightarrow{e_n} E_{n-1}$ is classified by a map $E_{n-1} \to K(\pi_n F, n+1)$, with the corresponding cohomology class $\kappa$ of $E_{n-1}$ referred to as the $k$-invariant in degree $n+1$ of the Moore--Postnikov system. The $k$-invariant pulls back to that of the Postnikov system of $F$, i.e. the diagram 

$$\begin{tikzcd}
F_n \arrow[r] \arrow[d, "f_n"] & E_n \arrow[d, "e_n"] \\
F_{n-1} \arrow[d] \arrow[r]    & E_{n-1} \arrow[d]    \\
{K(\pi_n F, n+1)} \arrow[r]    & {K(\pi_n F, n+1)}   
\end{tikzcd}$$ is homotopy commutative.

Since $E_n$ is the homotopy fiber of the map $E_{n-1} \to K(\pi_n F, n+1)$, we have $e_n^*(\kappa) = 0$. In fact, from the Serre long exact sequence we have that $\kappa$ generates the kernel of $H^{n+1}(E_{n-1}; \pi_n F) \xrightarrow{e_n^*} H^{n+1}(E_n; \pi_n F).$

\subsection{Naturality of the Moore--Postnikov system} Naturality of the ordinary Postnikov system (i.e. the Moore--Postnikov system of the map to a point) is discussed in detail in \cite{Ka63}. For an explicit treatment of naturality of the relative case, we point the reader to \cite[(5.3.5)]{Ba}; see also \cite[Section 3]{He60}.

Suppose we have a fibration $F^1 \to E^1 \xrightarrow{p_1} B^1$ mapping to another fibration $F^2 \to E^2 \xrightarrow{p_2} B^2$, i.e. a homotopy commutative diagram 

$$\begin{tikzcd}
F^1 \arrow[d] \arrow[r] & F^2 \arrow[d] \\
E^1 \arrow[r] \arrow[d, "p_1"] & E^2 \arrow[d, "p_2"] \\
B^1 \arrow[r]           & B^2          
\end{tikzcd}$$

Then, we can find Moore--Postnikov systems $\{E^1_n\}$ for $E^1 \to B^1$ and $\{E^2_n\}$ for $E^2 \to B^2$ and a homotopy commutative diagram 

\adjustbox{scale=0.8,center}{\begin{tikzcd}
F^1 \arrow[rd] \arrow[rrr] \arrow[d]     &                                          &                                & F^2 \arrow[rd] \arrow[d]     &                              &                 \\
\vdots \arrow[ddd]                       & E^1 \arrow[d] \arrow[rrr]                &                                & \vdots \arrow[ddd]           & E^2 \arrow[d]                &                 \\
                                         & \vdots \arrow[ddd]            &                                &                              & \vdots \arrow[ddd]           &                 \\
                                         &                                          &                                &                              &                              &                 \\
F_3^1 \arrow[rd] \arrow[rrr] \arrow[ddd] &                                          &                                & F_3^2 \arrow[rd] \arrow[ddd] &                              &                 \\
                                         & E_3^1 \arrow[rrr] \arrow[ddd] \arrow[rd] &                                &                              & E_3^2 \arrow[ddd] \arrow[rd] &                 \\
                                         &                                          & {K(\pi_4F^1,5)} \arrow[rrr]    &                              &                              & {K(\pi_4F^2,5)} \\
F_2^1 \arrow[rd] \arrow[rrr]             &                                          &                                & F_2^2 \arrow[rd]             &                              &                 \\
                                         & E_2^1 \arrow[ddd] \arrow[rrr] \arrow[rd] &                                &                              & E_2^2 \arrow[ddd] \arrow[rd] &                 \\
                                         &                                          & {K(\pi_3 F^1, 4) } \arrow[rrr] &                              &                              & {K(\pi_3F^2,4)} \\
                                         &                                          &                                &                              &                              &                 \\
                                         & B^1 \arrow[rrr] \arrow[rd]               &                                &                              & B^2 \arrow[rd]               &                 \\
                                         &                                          & {K(\pi_2F^1, 3)} \arrow[rrr]   &                              &                              & {K(\pi_2F^2,3)}
\end{tikzcd}}

\vspace{1em}

where the columns in the back are the induced Postnikov systems for $F^1$ and $F^2$.

\section{More preliminaries, and a proof of Ehresmann's theorem}

Let $X$ be a CW complex, and let $\xi$ be an oriented real rank $2n$ vector bundle over $X$, whose structure group we have reduced to $SO(2n)$ without loss of generality. Let $X \xrightarrow{f} BSO(2n)$ be the homotopy class of maps classifying this bundle up to isomorphism. The problem of determining whether $\xi$ admits an almost complex structure is equivalent to the problem of lifting $f$ through the natural map $BU(n) \to BSO(2n)$.

$$\begin{tikzcd}
                                         &  & BU(n) \arrow[d] \\
X \arrow[rr, "f"] \arrow[rru, dashed] &  & BSO(2n)        
\end{tikzcd}$$

This is an obstruction-theoretic problem, which we address by decomposing $BU(n) \to BSO(2n)$ into a Moore--Postnikov system and determining the obstructions to lifting stage-by-stage.

\subsection{Some relevant homotopy groups} We list the relevant homotopy groups of $SO(2n)/U(n)$ for our lifting problem here. First of all, since $U(n+1)$ acts transitively on $S^{2n+1}$ with isotropy $U(n)$, we have the fiber bundle $U(n) \to U(n+1) \to S^{2n+1}$, giving us a fibration $S^{2n+1} \to BU(n) \to BU(n+1)$. From here one sees that the natural map $BU(n) \to BU(n+1)$ induces an isomorphism on $\pi_{\leq 2n}$. Hence the homotopy groups $\pi_{\leq 2n} BU(n) $ are stable, and in particular isomorphic, via the map induced by the inclusion, to $\pi_{\leq 2n} BU$. The homotopy groups of $BU$ are known by Bott's periodicity theorem: concretely, $\pi_{2k}BU \cong \mathbb{Z}$ and $\pi_{2k+1}BU = 0$ for all $k$. Similarly, one obtains that the groups $\pi_{\leq 2n-1}BSO(2n)$ are stable. The homotopy groups of the stable space $BO$ are eightfold-periodic: starting from $\pi_1 BO$ they read $\mathbb{Z}_2, \mathbb{Z}_2, 0, \mathbb{Z}, 0, 0, 0, \mathbb{Z}$. The homotopy groups of $BSO$ differ only in $\pi_1$, where we have $\pi_1 BSO = 0$. 

From the map of fibrations $$\begin{tikzcd} SO(2n)/U(n) \arrow[r] \arrow[d] & SO/U \arrow[d] \\  BU(n) \arrow[r] \arrow[d] & BU \arrow[d] \\ BSO(2n) \arrow[r] & BSO \end{tikzcd} $$ and the induced map of long exact sequences in homotopy groups, we see that the groups $\pi_{\leq 2n-2} SO(2n)/U(n)$ are stable using the five lemma, i.e. they are isomorphic via the natural map to those of $SO/U$. A part of Bott's periodicity theorem also gives us that $SO/U$ has the homotopy type of a connected component of the based loop space $\Omega O$. Hence its first several homotopy groups, starting with $\pi_1$, are given by $$0, \mathbb{Z}, 0, 0, 0, \mathbb{Z}, \mathbb{Z}_2, \mathbb{Z}_2, 0, \mathbb{Z}, 0, 0, 0, \mathbb{Z}, \ldots $$ In particular, we have $$\pi_{4k+2} SO/U \cong \mathbb{Z}.$$

The first unstable homotopy group of $SO(2n)/U(n)$, namely $\pi_{2n-1} SO(2n)/U(n)$, depends on the mod 4 class of $n$, and is given by 
\begin{equation}\label{harris}\pi_{2n-1}(SO(2n)/U(n)) =
\begin{cases}
\mathbb{Z} \oplus \mathbb{Z}_2 & n\equiv 0 \bmod 4,\\
\mathbb{Z}_{(n-1)!} & n \equiv 1 \bmod 4,\\
\mathbb{Z} & n \equiv 2 \bmod 4,\\
\mathbb{Z}_{\tfrac{1}{2}(n-1)!} & n \equiv 3 \bmod 4.
\end{cases}
\end{equation}

This is summarized in \cite[Lemma 1]{M61}, and the reference to (then forthcoming) work of Bruno Harris therein is to \cite{Ha63}.

\subsection{The first obstruction to an almost complex structure} Since $SO(2)/U(1)$ is a point, there are no obstructions to reducing the structure group. From now on we suppose $n\geq 2$. Then $\pi_2 SO(2n)/U(n)$ is stable, and hence isomorphic to $\mathbb{Z}$. Therefore, the first $k$-invariant in the Moore--Postnikov system for $BU(n) \to BSO(2n)$ is a cohomology class in $H^3(BSO(2n); \mathbb{Z})$.

Since $BSO(2n)$ is simply connected, by the Hurewicz theorem we have that $$H_2(BSO(2n);\mathbb{Z}) \cong \pi_2(BSO(2n); \mathbb{Z}) \cong \mathbb{Z}_2,$$ and that $\pi_3(BSO(2n)) = 0$ surjects onto $H_3(BSO(2n);\mathbb{Z})$, i.e. the latter group is trivial. By the universal coefficient theorem, we have $H^3(BSO(2n);\mathbb{Z}) \cong \mathbb{Z}_2$. The non-zero element in this group is the third integral Stiefel--Whitney class $W_3 = \beta w_2$. Here and throughout, $\beta$ denotes the Bockstein morphism associated to the short exact sequence $0 \to \ZZ \to \ZZ \to \ZZ_2 \to 0$.




Denote by $E_2$ the second stage of the Moore--Postnikov system for $BU(n) \to BSO(2n)$, as in \Cref{moorepostnikov}. Since $BU(n) \to E_2$ has two--connected fiber, it induces an injection on $H^3(-;\mathbb{Z})$. As $H^3(BU(n);\mathbb{Z}) = 0$, we have $H^3(E_2;\ZZ) = 0$. The $k$-invariant for the fibration $E_2 \to BSO(2n)$ generates the kernel of the induced map on $H^3(-;\ZZ)$, and therefore it is the class $W_3$. We note that $W_3$ of an orientable bundle vanishes precisely when $w_2$ has an integral lift, i.e. when $w_2$ is the mod 2 reduction of an integral cohomology class, which would be the first Chern class of the almost complex structure we would obtain if we could lift to $BU(n)$. The homotopy fiber of $BSO(2n) \xrightarrow{W_3} K(\mathbb{Z}, 3)$ can be identified with $BSpin^c(2n)$, the classifying space of the Lie group $Spin^c(2n)$, though we will not make use of this here.

Note that $W_3$ of the tangent bundle of an  orientable four-manifold vanishes, see \cite{TV}. This is not true of higher dimensional manifolds: for example, the five-dimensional Wu manifold $M := SU(3)/SO(3)$ has non-zero $W_3$. In particular, the six-manifold $S^1\times M$ does not admit an almost complex structure. 

\subsection{The second obstruction to an almost complex structure}

Supposing $W_3$ of the bundle is $0$, we now have to distinguish three cases: $n=2$, $n=3$, and $n\geq 4$. If $n=2$, then the second (and, in the case of $X$ having dimension $\leq 4$, last) obstruction to lifting to $BU(2)$ lies in $H^4(X; \pi_3 SO(4)/U(2)) \cong H^4(X; \mathbb{Z})$. Since $W_3 = 0$ we can choose a lift $X \to E_2$ of the map $\tau$, which yields an integral lift $c$ of $w_2$ on $X$. The second obstruction is then obtained by pulling back the $k$-invariant $E_2 \to K(\mathbb{Z}, 4)$ to $X$ via this lift (see e.g. \cite[Chapter IX.7]{Wh}). Different choices of lift will generally give different classes in $H^4(X; \mathbb{Z})$. A theorem of Wu \cite[Théorème IV.10]{W52} tells us that, on a closed orientable four--manifold, this obstruction class vanishes if and only if $c^2 - 2e - p_1$ vanishes. More generally, for an $SO(4)$ bundle over a CW complex, $c^2 - 2e - p_1$ is four times this second obstruction; see \Cref{warmupThmII}.

For $n=3$, we have $SO(6)/U(3) = \mathbb{CP}^3$, so the second obstruction lies in $H^8(X;\pi_7 \mathbb{CP}^3) \cong H^8(X;\ZZ)$, and we will address this case later, in \Cref{secondobstruction2n=6}. In particular, if $X$ is a six--dimensional CW complex, this and all further obstructions vanish, leaving $W_3$ as the sole obstruction, recovering an observation of Ehresmann \cite[section 6]{E52}. Let us now focus on the case $n \geq 4$, where the second obstruction to lifting to $BU(n)$ is stable, and lies in $H^7(X;\pi_6 SO(2n)/U(n)) \cong H^7(X; \mathbb{Z})$. By the obstruction being stable, we mean that it is obtained by pulling back the corresponding obstruction for lifting the stabilized real bundle $X \to BSO$ through $BU \to BSO$, where one has the commutative diagram $$\begin{tikzcd}
SO(2n)/U(n) \arrow[d] \arrow[r] & SO/U \arrow[d] \\
BU(n) \arrow[d] \arrow[r]       & BU \arrow[d]   \\
BSO(2n) \arrow[r]               & BSO           
\end{tikzcd}$$ We will not make use of this observation so we will not go into the details.

As with $W_3$, we see that since $W_7$ pulls back to the zero class in $BU(n)$, we must have $W_7 = 0$ for a complex vector bundle (since $c_3$ will be an integral lift of $w_6$). However, it may be the case that when pulling back $W_7$ to $E_2$, the class becomes divisible by an integer $\geq 2$, and the relevant $k$-invariant is in fact a ``fractional part'' of $W_7$. Even though this does not happen in the case of $W_7$ as we will now see, later on we will observe that this phenomenon does occur in higher degrees, for $W_{11}, W_{15}, \ldots$ One might also find it curious that $W_5$ has been skipped; as soon as $W_3 = 0$, we have that $W_5 = 0$. Let us record a proof of this phenomenon here; the method will also become relevant again in \Cref{TheoremIIsection}.

\begin{proposition}\label{integralSW} Let $\xi \to X$ be an orientable real vector bundle such that $W_{4k-1} = 0$ for $k\leq m$. Then $W_{4m+1} = 0$. \end{proposition}

We will use the following formula \cite[Theorem C]{T60} attributed to Wu\footnote{In the notation of loc. cit., $W_i$ denotes the mod 2 Stiefel--Whitney class $w_i$.}, involving the Pontryagin square operation $H^{2m}(X;\mathbb{Z}_2) \xrightarrow{\mathfrak{P}} H^{4m}(X;\mathbb{Z}_4)$, which for orientable bundles (i.e. $w_1=0$) becomes

\begin{equation}\label{psquare} \mathfrak{P}(w_{2m}) = \rho_4(p_m) + \theta_2\left( \sum_{j=0}^{m-1}w_{2j}w_{4m-2j} \right).\end{equation}

Here, $p_m$ denotes the Pontryagin class, $\rho_4$ denotes mod 4 reduction, and $\theta_2$ denotes the map on cohomology $H^{4m}(X;\mathbb{Z}_2) \to H^{4m}(X;\mathbb{Z}_4)$ induced by the inclusion $\mathbb{Z}_2 \hookrightarrow \mathbb{Z}_4$. Before proceeding with the proof of \Cref{integralSW}, we record the following:

\begin{lemma}\label{pontryaginsquarelemma} Let $\tilde{u}$ be an integral lift of a mod 2 cohomology class $u$. Then $\mathfrak{P}(u) = \rho_4(\tilde{u}^2)$. \end{lemma}

\begin{proof} By construction, the Pontryagin square satisfies $$\mathfrak{P}(\rho_2^4(x)) = x^2$$ for any mod 4 class $x$, where $\rho_2^4$ denotes the mod 2 map $H^*(-;\ZZ_4) \to H^*(-;\ZZ_2)$; see e.g. \cite[(1.3)]{BT62}. Note that $\rho_2^4(\rho_4(\tilde{u})) = \rho_2(\tilde{u}) = u$.  Therefore \[\mathfrak{P}(u) = \mathfrak{P}(\rho_2^4(\rho_4(\tilde{u}))) = \rho_4(\tilde{u})^2 = \rho_4(\tilde{u}^2).\hfill \qedhere\] \end{proof}

\begin{proof}[Proof of \Cref{integralSW}]

Since the classes $W_{\leq 4m-1}$ vanish, the classes $w_{\leq 4m-2}$ are the mod 2 reductions of integral classes. Let us denote by $c_k$ some fixed integral lift of $w_{2k}$, for $k\leq 2m-1$. Since $c_m$ is an integral lift of $w_{2m}$, \Cref{pontryaginsquarelemma} gives us $$\mathfrak{P}(w_{2m}) = \rho_4(c_m^2).$$

Now, consider the following commutative diagram of short exact sequences of abelian groups

$$\begin{tikzcd}
0 \arrow[r] \arrow[d] & \ZZ \arrow[d, "\mathrm{mod }\,  2"] \arrow[r] & \ZZ \arrow[d, "\mathrm{mod }\,  4"] \arrow[r] & \ZZ_2 \arrow[d, "="] \arrow[r] & 0 \arrow[d] \\
0 \arrow[r]           & \ZZ_2 \arrow[r]         & \ZZ_4 \arrow[r]         & \ZZ_2 \arrow[r]           & 0          
\end{tikzcd}$$

and the following commutative diagram coming from the induced map of long exact sequences in cohomology $$\begin{tikzcd}
H^{4m-1}(X;\ZZ_2) \arrow[r, "\beta"] \arrow[d, "="] & H^{4m}(X;\ZZ) \arrow[d, "\rho_2"] \arrow[r, "\times 2"] & H^{4m}(X;\ZZ) \arrow[d, "\rho_4"] \arrow[r, "\rho_2"] & H^{4m}(X;\ZZ_2) \arrow[d, "="]  \\
H^{4m-1}(X;\ZZ_2) \arrow[r, "Sq^1"]   & H^{4m}(X;\ZZ_2) \arrow[r, "\theta_2"]       & H^{4m}(X;\ZZ_4) \arrow[r, "\rho_2^4"]              & H^{4m}(X;\ZZ_2)               
\end{tikzcd}$$

In particular, from the commutativity of the middle square, for any class $x$ we have \begin{equation}\label{2x-identity} \theta_2(\rho_2(x)) = \rho_4(2x).\end{equation} Now, since $\sum_{j=1}^{m-1} w_{2j}w_{4m-2j} = \rho_2\left( \sum_{j=1}^{m-1} c_j c_{2m-j} \right)$, we conclude that 
$$\theta_2 \left(\sum_{j=1}^{m-1} w_{2j}w_{4m-2j}\right) = \rho_4\left( 2\sum_{j=1}^{m-1} c_j c_{2m-j}\right).$$

Recall $\mathfrak{P}(w_{2m}) = \rho_4(c_m^2)$. Since $\theta_2(w_{4m}) = \mathfrak{P}(w_{2m}) - \rho_4(p_m) - \theta_2\left( \sum_{j=1}^{m-1} w_{2j}w_{4m-2j}\right),$ we therefore have $$\theta_2(w_{4m}) = \rho_4\left( c_m^2 - p_m - 2\sum_{j=1}^{m-1} c_j c_{2m-j} \right).$$ Since $\rho_2^4(\theta_2(w_{4m})) = 0$ by exactness, we have $$\rho_2\left( c_m^2 - p_m - 2\sum_{j=1}^{m-1} c_j c_{2m-j}\right) = 0$$ using the right-most square of the diagram. Therefore, there is a class $x \in H^{4m}(X;\ZZ)$ such that $$2x = c_m^2 - p_m - 2\sum_{j=1}^{m-1} c_j c_{2m-j}.$$ Then $\theta_2(\rho_2(x)) = \rho_4(2x) = \theta_2(w_{4m})$. Therefore, $\rho_2(x) + w_{4m} \in \ker(\theta_2) = \mathrm{im}(Sq^1)$. Hence, there is some $y \in H^{4m-1}(X;\ZZ_2)$ such that $Sq^1(y) = \rho_2(x) + w_{4m}$. Since $Sq^1 = \rho_2 \circ \beta$, we have $$w_{4m} = \rho_2(x + \beta(y)).$$ That is, $x+\beta(y)$ is an integral lift of $w_{4m}$, and hence $W_{4m+1} = 0$. \end{proof}

We return to the second obstruction to lifting through $BU(n) \to BSO(2n)$, for $n\geq 4$. We will identify it as $W_7$ in another way in \Cref{theoremIsection}, as part of proving the more general \Cref{TheoremI}. However, here we address it directly, as Massey points out \cite[Remark 1]{M61}, this obstruction was known earlier, yet stated without proof, by Ehresmann \cite{E50}.

Denote by $p$ the map $E_2 \to BSO(2n)$. To show that the $k$-invariant $E_2 \to K(\mathbb{Z}, 7)$ is $p^*W_7$, we calculate the relevant cohomology of $E_2$ from the Serre spectral sequence associated to the fibration $K(\mathbb{Z}, 2) \to E_2 \to BSO(2n)$. For the integral cohomology of $BSO(2n)$, we point the reader to \cite[Theorem 1.6]{B82} for a detailed treatment. Below we have the relevant terms of the second page of the spectral sequence for $n\geq 5$. The only difference in the case of $n=4$ is that the Euler class appears in degree eight. However, it does not affect our computation, and so our conclusion will hold in this case as well.

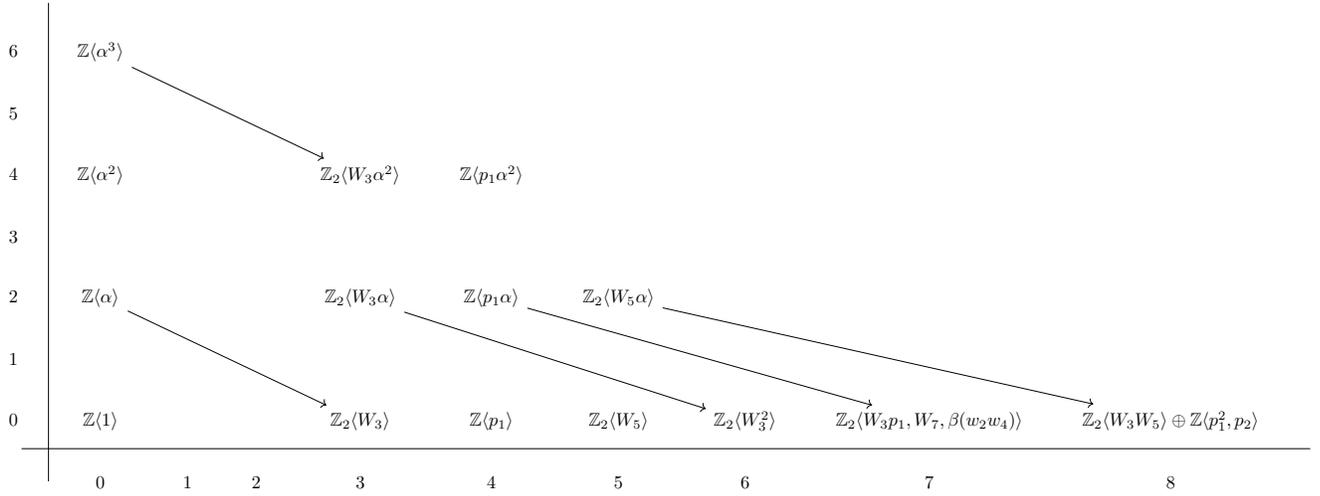
\adjustbox{scale=0.69,center}{
\begin{tikzcd}
{}                                                                                    &                                                  &   &   &                                                      &                                                   &                                                     &                                   &                                                           &                                    &    \\
6                                                                                     & \mathbb{Z}\langle \alpha^3 \rangle \arrow[rrrdd] &   &   &                                                      &                                                   &                                                     &                                   &                                                           &                                    &    \\
5                                                                                     &                                                  &   &   &                                                      &                                                   &                                                     &                                   &                                                           &                                    &    \\
4                                                                                     & \large \mathbb{Z}\langle \alpha^2 \rangle        &   &   & \mathbb{Z}_2\langle W_3\alpha^2 \rangle              & \mathbb{Z}\langle p_1\alpha^2 \rangle             &                                                     &                                   &                                                           &                                    &    \\
3                                                                                     &                                                  &   &   &                                                      &                                                   &                                                     &                                   &                                                           &                                    &    \\
2                                                                                     & \mathbb{Z}\langle \alpha \rangle \arrow[rrrdd]   &   &   & \mathbb{Z}_2\langle W_3 \alpha \rangle \arrow[rrrdd] & \mathbb{Z}\langle p_1\alpha \rangle \arrow[rrrdd] & \mathbb{Z}_2\langle W_5\alpha \rangle \arrow[rrrdd] &                                   &                                                           &                                    &    \\
1                                                                                     &                                                  &   &   &                                                      &                                                   &                                                     &                                   &                                                           &                                    &    \\
0                                                                                     & \mathbb{Z}\langle 1 \rangle                      &   &   & \mathbb{Z}_2\langle W_3 \rangle                      & \mathbb{Z}\langle p_1 \rangle                     & \mathbb{Z}_2\langle W_5 \rangle                     & \mathbb{Z}_2\langle W_3^2 \rangle & {\mathbb{Z}_2\langle W_3 p_1, W_7, \beta(w_2w_4) \rangle} & \mathbb{Z}_2\langle W_3W_5 \rangle \oplus \mathbb{Z}\langle p_1^2, p_2 \rangle &    \\
{} \arrow[uuuuuuuu, no head, shift right=8] \arrow[rrrrrrrrrr, no head, shift left=8] & 0                                                & 1 & 2 & 3                                                    & 4                                                 & 5                                                   & 6                                 & 7                                                         & 8                                  & {}
\end{tikzcd}
}

\vspace{1em}

Note that, since the cohomology of the fiber is concentrated in even degrees, only the odd-indexed pages of the spectral sequence can have non-trivial differentials. On the third page, since $p^*W_3 = 0$, we must have $d_3(\alpha) = W_3$. Note that then $d_3(\alpha^2) = 2\alpha W_3 = 0$, so $\alpha^2$ together with $2\alpha$ survives to the fourth, and hence fifth, page. We have $d_3(\alpha^3) = 3W_3\alpha^2 = W_3\alpha^2, \,\,\,\, d_3(W_3 \alpha) = W_3^2$, and $d_3(W_5\alpha) = W_3W_5$. Moving on to the fifth page of the spectral sequence, we have $d_5(\alpha^2) = W_5$:

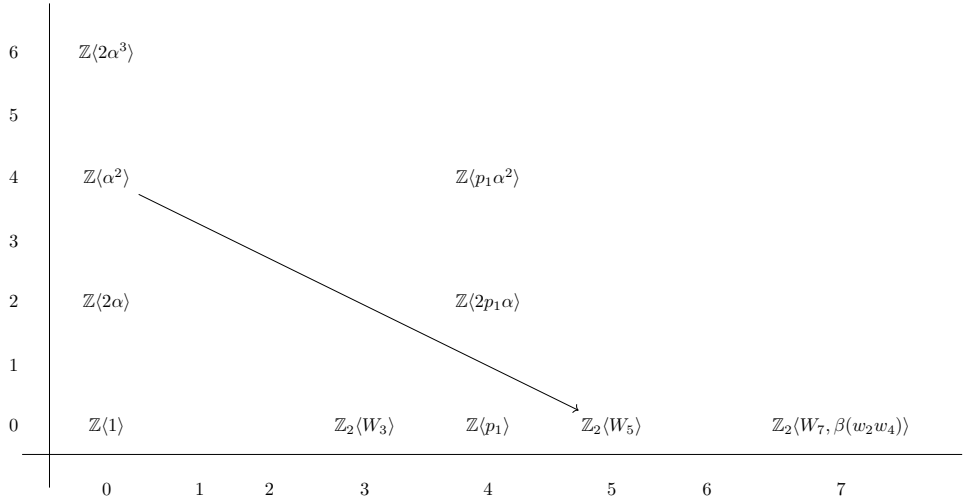
\adjustbox{scale=0.7,center}{
\begin{tikzcd}
{}                                                                                   &                                                      &   &   &                                 &                                       &                                 &   &                                                  &    \\
6                                                                                    & \mathbb{Z}\langle 2\alpha^3 \rangle                  &   &   &                                 &                                       &                                 &   &                                                  &    \\
5                                                                                    &                                                      &   &   &                                 &                                       &                                 &   &                                                  &    \\
4                                                                                    & \mathbb{Z}\langle \alpha^2 \rangle \arrow[rrrrrdddd] &   &   &                                 & \mathbb{Z}\langle p_1\alpha^2 \rangle &                                 &   &                                                  &    \\
3                                                                                    &                                                      &   &   &                                 &                                       &                                 &   &                                                  &    \\
2                                                                                    & \mathbb{Z}\langle 2\alpha \rangle                    &   &   &                                 & \mathbb{Z}\langle 2 p_1\alpha \rangle &                                 &   &                                                  &    \\
1                                                                                    &                                                      &   &   &                                 &                                       &                                 &   &                                                  &    \\
0                                                                                    & \mathbb{Z}\langle 1 \rangle                          &   &   &  & \mathbb{Z}\langle p_1 \rangle         & \mathbb{Z}_2\langle W_5 \rangle &   & {\mathbb{Z}_2\langle W_7, \beta(w_2w_4) \rangle} &    \\
{} \arrow[uuuuuuuu, no head, shift right=8] \arrow[rrrrrrrrr, no head, shift left=8] & 0                                                    & 1 & 2 & 3                               & 4                                     & 5                               & 6 & 7                                                & {}
\end{tikzcd}
}

\vspace{1em}

Indeed, first of all, since the third, fourth, and fifth homotopy groups of $SO(2n)/U(n)$ are trivial, we have that $E_5 \xrightarrow{e_5} E_4 \xrightarrow{e_4} E_3 \xrightarrow{e_3} E_2$ is a composition of homotopy equivalences. Therefore, $BU(n) \to E_2 \simeq E_5$ is an isomorphism on $H^{\leq 5}(-;\mathbb{Z})$. Since $H^5(BU(n);\mathbb{Z}) = 0$, we have $p^*W_5 = 0$ and hence it must be at this page that $W_5$ is killed, i.e. $d_5(\alpha^2) = p^*W_5$. On the seventh page, we have $d_7(2\alpha^3) = p^*\beta(w_2w_4)$.


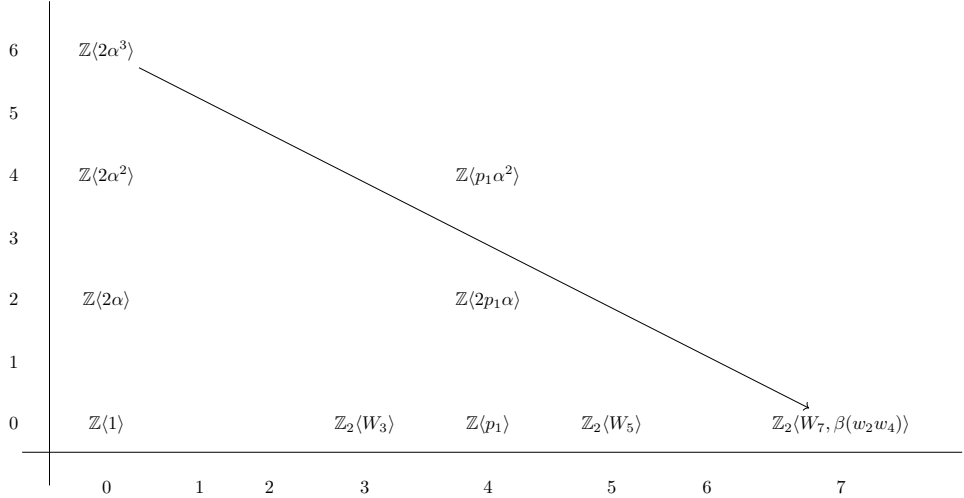
\adjustbox{scale=0.7, center}{
\begin{tikzcd}
{}                                                                                   &                                                           &   &   &                                 &                                       &                                 &   &                                                  &    \\
6                                                                                    & \mathbb{Z}\langle 2\alpha^3 \rangle \arrow[rrrrrrrdddddd] &   &   &                                 &                                       &                                 &   &                                                  &    \\
5                                                                                    &                                                           &   &   &                                 &                                       &                                 &   &                                                  &    \\
4                                                                                    & \mathbb{Z}\langle 2\alpha^2 \rangle                       &   &   &                                 & \mathbb{Z}\langle p_1\alpha^2 \rangle &                                 &   &                                                  &    \\
3                                                                                    &                                                           &   &   &                                 &                                       &                                 &   &                                                  &    \\
2                                                                                    & \mathbb{Z}\langle 2\alpha \rangle                         &   &   &                                 & \mathbb{Z}\langle 2 p_1\alpha \rangle &                                 &   &                                                  &    \\
1                                                                                    &                                                           &   &   &                                 &                                       &                                 &   &                                                  &    \\
0                                                                                    & \mathbb{Z}\langle 1 \rangle                               &   &   &  & \mathbb{Z}\langle p_1 \rangle         &  &   & {\mathbb{Z}_2\langle W_7, \beta(w_2w_4) \rangle} &    \\
{} \arrow[uuuuuuuu, no head, shift right=8] \arrow[rrrrrrrrr, no head, shift left=8] & 0                                                         & 1 & 2 & 3                               & 4                                     & 5                               & 6 & 7                                                & {}
\end{tikzcd}
}

\vspace{2em}

Indeed, since $p^*W_3 = p^*W_5 = 0$, the classes $p^*w_2$ and $p^*w_4$ have integral lifts in $H^*(E_2;\mathbb{Z})$. Hence $p^*(w_2w_4)$ has an integral lift as well, so $p^*\beta(w_2w_4) = \beta(p^*(w_2w_4)) = 0$. The only possible differential that could kill $\beta(w_2w_4)$ is the one above on the seventh page.

We conclude that $H^7(E_2;\mathbb{Z}) = \{0, p^*W_7\}$. Since $BU(n) \to E_6$ induces an injection on $H^7(-;\mathbb{Z})$, it must be that $p^*W_7$ pulls back to be zero in $E_6$, and $W_7$ vanishes when pulled back to $BU(n)$. We conclude that the $k$-invariant $E_2 \to K(\mathbb{Z}, 7)$ is $p^*W_7$:

\begin{thm}(Ehresmann, \cite[Section 6]{E50}) The second obstruction to reducing the structure group of an $SO(2n)$ bundle to $U(n)$, where $n\geq 4$, is $W_7$. \end{thm}

Note that if $n=3$, then $W_7$ automatically vanishes, as the Euler class is an integral lift of $w_6$.

\section{Proof of Massey's Theorem I}\label{theoremIsection}

We now provide the details for Massey's generalization of the above result on $W_7$:

\begin{thm}\label{TheoremI}\cite[Theorem I]{M61} Let $\theta \to X$ be the $SO(2n)/U(n)$ bundle associated to an oriented rank $2n$ real vector bundle $\xi$ over a CW complex $X$, and let $s$ be a section of the bundle over the $(4k+2)$--skeleton of $X$. If $4k+3 < 2n$, then $$W_{4k+3}(\xi) = \begin{cases}
(2k)!\, \mathfrak{o}_{4k+3} \ \ \ \ &k \, \mathrm{even},\\
\tfrac{1}{2}(2k)!\, \mathfrak{o}_{4k+3} \ \ \ \  &k \, \mathrm{odd},
\end{cases}$$
where $\mathfrak{o}_{4k+3}$ denotes the obstruction\footnote{Note that the obstruction is only well-defined up to sign, and we make a particular choice in the above statement.} to extending $s$ over the $(4k+3)$--skeleton of $X$. \end{thm}

Before embarking on the proof, we want to point out that the assumption $4k+3 < 2n$ will be used (in the proof of \Cref{lemma}(b)) to identify the homotopy groups of $SO(2n)/U(n)$ in degrees $\leq 4k+2$ with those of $SO/U$ which are known by Bott periodicity. We want to relate the obstruction $\mathfrak{o} = \mathfrak{o}_{4k+3}$, which is the pullback of a $k$--invariant in the Moore--Postnikov system for $BU(n)\to BSO(2n)$, to the integral Stiefel--Whitney class $W_{4k+3}$. The latter naturally appears as a $k$--invariant in the Moore--Postnikov system for $BSO(4k+2) \to BSO(2n)$ as the primary obstruction to finding $2n-4k-2$ linearly independent sections of a rank $2n$ oriented real bundle \cite[\textsection 38]{St}. Now, there is no natural  map between $U(n)$ and $SO(4k+2)$ in either direction, so we consider an intermediate space, the intersection of $U(n)$ and $SO(4k+2)$, namely $U(2k+1)$. The inclusion of $U(2k+1)$ into $U(n)$ and $SO(4k+2)$ gives us the following commutative diagram:

\begin{equation}\label{maindiagram}
\begin{tikzcd}
SO(2n)/U(n) \arrow[d] &  &                                                             &  & SO(2n)/SO(4k+2) \arrow[d] \\
BU(n) \arrow[d]       &  & SO(2n)/U(2k+1) \arrow[rru, "r"] \arrow[llu, "q"'] \arrow[d] &  & BSO(4k+2) \arrow[d]       \\
BSO(2n)               &  & BU(2k+1) \arrow[d] \arrow[rru] \arrow[llu]                  &  & BSO(2n)                   \\
                      &  & BSO(2n) \arrow[rru, "id"] \arrow[llu, "id"']                &  &                          
\end{tikzcd}\end{equation}

We focus for the moment on the induced maps $q$ and $r$ on fibers. 

\begin{lemma}\label{lemma} We have the following: \begin{enumerate} \item[(a)] The map $q$ induces an isomorphism on $\pi_{\leq 4k+2}$. 
\item[(b)] $\pi_{4k+2}$ of each of $SO(2n)/U(n), SO(2n)/U(2k+1)$, and $SO(2n)/SO(4k+2)$ is isomorphic to $\mathbb{Z}$.
\item[(c)] $SO(2n)/SO(4k+2)$ is $(4k+1)$--connected, and $r$ induces multiplication by $\ell$ on $\pi_{4k+2}$, up to sign, where $\ell = (2k)!$ if $k$ is even, and $\ell = \tfrac{1}{2}(2k)!$ if $k$ is odd. \end{enumerate} \end{lemma}

\begin{proof} (a) We have the fiber bundle $$U(n)/U(2k+1) \to SO(2n)/U(2k+1) \xrightarrow{q} SO(2n)/U(n).$$ The complex Stiefel manifold $U(n)/U(2k+1)$ is $(4k+2)$--connected, so $q$ induces an isomorphism on $\pi_{\leq 4k+2}$.

(b) Since $4k+2 \leq 2n-2$, we have $\pi_{4k+2}SO(2n)/U(n) \cong \pi_{4k+2}SO/U \cong \mathbb{Z}$. From part (a) we conclude that $\pi_{4k+2} SO(2n)/U(2k+1) \cong \mathbb{Z}$. Finally, the Stiefel manifold $SO(2n)/SO(4k+2)$ is well known to be $(4k+1)$--connected with $\pi_{4k+2} SO(2n)/SO(4k+2) \cong \mathbb{Z}$.

(c) We have the fiber bundle \begin{equation}\label{mainfiber} SO(4k+2)/U(2k+1) \to SO(2n)/U(2k+1) \xrightarrow{r} SO(2n)/SO(4k+2) \end{equation} and the following part of the associated long exact sequence in homotopy groups: \begin{equation*}
\text{\footnotesize $\pi_{4k+2} SO(2n)/U(2k+1) \xrightarrow{r_*} \pi_{4k+2} SO(2n)/SO(4k+2) \to \pi_{4k+1} SO(4k+2)/U(2k+1) \to \pi_{4k+1} SO(2n)/U(2k+1)$}
\end{equation*}

By part (b), the first two groups are isomorphic to $\mathbb{Z}$. The third group is the first unstable homotopy group of $SO(4k+2)/U(2k+1)$ listed in \eqref{harris}. Namely, it is isomorphic to  $\mathbb{Z}_\ell$, where $\ell = (2k)!$ if $k$ is even, and $\ell = \tfrac{1}{2}(2k)!$ if $k$ is odd. The fourth group is, by part (a), isomorphic to $\pi_{4k+1}SO(2n)/U(n) \cong \pi_{4k+1} SO/U = 0$, as $4k+1 < 2n-2$. Therefore, the above exact sequence is given by $$\mathbb{Z} \xrightarrow{r_*} \mathbb{Z} \to \mathbb{Z}_\ell \to 0,$$ and so the induced map $\mathbb{Z} \xrightarrow{r_*} \mathbb{Z}$ on $\pi_{4k+2}$ is, up to sign, multiplication by $\ell$. \end{proof}




Consider now the following subdiagrams of diagram \eqref{maindiagram}:

\begin{equation}\label{q} \begin{tikzcd}
BU(2k+1) \arrow[r] \arrow[d] & BU(n) \arrow[d] \\
BSO(2n) \arrow[r, "id"]            & BSO(2n)        
\end{tikzcd}\end{equation} inducing the map $SO(2n)/U(2k+1) \xrightarrow{q} SO(2n)/U(n)$ on fibers, and \begin{equation}\label{r} \begin{tikzcd}
BU(2k+1) \arrow[d] \arrow[r] & BSO(4k+2) \arrow[d] \\
BSO(2n) \arrow[r, "id"]            & BSO(2n)            
\end{tikzcd}\end{equation} inducing the map $SO(2n)/U(2k+1) \xrightarrow{r} SO(2n)/SO(4k+2)$ on fibers. Let $$F^1 = SO(2n)/U(2k+1), \,\,\, \, \, F^2 = SO(2n)/U(n), \,\,\, \, \, F^3 = SO(2n)/SO(4k+2)$$ so that $F^1 \xrightarrow{q} F^2$ and $F^1 \xrightarrow{r} F^3$. 

\subsection*{Induced lift and effect on the $k$-invariant in diagram \eqref{q}} Consider the induced map of Moore--Postnikov systems corresponding to diagram \eqref{q}:

$$\begin{tikzcd}
F_2^1 \arrow[rd] & \vdots \arrow[d]           & F_2^2 \arrow[rd] & \vdots \arrow[d] \\
                 & E_2^1 \arrow[d] \arrow[rr] &                  & E_2^2 \arrow[d]  \\
                 & BSO(2n) \arrow[rr]         &                  & BSO(2n)         
\end{tikzcd}$$

Since $F^1 \xrightarrow{q} F^2$ induces isomorphisms on $\pi_{\leq 4k+2}$ by \Cref{lemma}(a), from the commutative diagram $$\begin{tikzcd}
F^1 \arrow[d, "g^1_s"] \arrow[r, "q"] & F^2 \arrow[d, "g^2_s"] \\
F^1_s \arrow[r]              & F^2_s        
\end{tikzcd}$$

and from $\pi_{>s}F^i_s = 0$, $i=1,2$ (recall \Cref{ordinarycase}), we see that the maps $F^1_s \to F^2_s$ are weak homotopy equivalences for $s \leq 4k+2$. From the induced map of long exact sequences in homotopy groups associated to the diagram $$\begin{tikzcd}
F_s^1 \arrow[r] \arrow[d] & F_s^2 \arrow[d] \\
E_s^1 \arrow[r] \arrow[d] & E_s^2 \arrow[d] \\
E_{s-1}^1 \arrow[r]       & E_{s-1}^2      
\end{tikzcd}$$ repeated application of the five lemma shows that the map $E_s^1 \to E_s^2$ is a weak homotopy equivalence for $s \leq 4k+2$. 

By assumption, we have a lift of the map $X \to BSO(2n)$ up to $E_{4k+1}^2$. Choosing a homotopy inverse of $E_{4k+1}^1 \to E_{4k+1}^2$ then gives us a lift to $E^1_{4k+1}$ in the left hand Moore--Postnikov system.

Our obstruction $\mathfrak{o}$ in the statement of \Cref{TheoremI} is (the pull back to $X$ of) the $k$--invariant, which takes the form $E_{4k+1}^2 \to K(\ZZ, 4k+3)$ by \Cref{lemma}(b). Since the homotopy fiber of a $k$--invariant in a Moore--Postnikov system is the next stage in the system, from the induced map on long exact sequences in homotopy groups for the diagram

$$\begin{tikzcd}
E_{4k+2}^1 \arrow[rr] \arrow[d]  &                                  & E_{4k+2}^2 \arrow[d]  &                       \\
E_{4k+1}^1 \arrow[rr] \arrow[rd] &                                  & E_{4k+1}^2 \arrow[rd, "\mathfrak{o}"] &                       \\
                                 & {K(\mathbb{Z}, 4k+3)} \arrow[rr] &                       & {K(\mathbb{Z}, 4k+3)}
\end{tikzcd}$$

we have that the bottom map $K(\ZZ, 4k+3) \to K(\ZZ, 4k+3)$ is a homotopy equivalence by the five lemma. Therefore, it induces an isomorphism $H^{4k+3}(X;\ZZ) \to H^{4k+3}(X;\ZZ)$; let us denote the preimage of $\mathfrak{o}$ under this isomorphism simply by $\mathfrak{o}$.

\subsection*{Induced lift and effect on the $k$-invariant in diagram \eqref{r}} Now we turn to the map of Moore--Postnikov systems induced by the diagram \eqref{r}. This time, as opposed to diagram \eqref{q}, the induced map of homotopy fibers is not highly connected. Regardless, our chosen lift of $X \to BSO(2n)$ through $E_{4k+1}^1$ will induce a lift of $X \to BSO(2n)$ through $E_{4k+1}^3$ by post-composition. (In fact, this holds for an even more immediate reason, namely that $E_{4k+1}^3 \to BSO(2n)$ is a homotopy equivalence. For uniformity of notation we will put this observation on hold for now.)

$$\begin{tikzcd}
F_{4k+2}^1 \arrow[rd] \arrow[rr] &                                             & F_{4k+2}^3 \arrow[rd]            &                                  &                       \\
                                 & E_{4k+2}^1 \arrow[dd] \arrow[rr]            &                                  & E_{4k+2}^3 \arrow[dd]            &                       \\
                                 &                                             &                                  &                                  &                       \\
                                 & E_{4k+1}^1 \arrow[rr] \arrow[rd] \arrow[dd] &                                  & E_{4k+1}^3 \arrow[rd] \arrow[dd] &                       \\
                                 &                                             & {K(\mathbb{Z}, 4k+3)} \arrow[rr] &                                  & {K(\mathbb{Z}, 4k+3)} \\
                                 & \vdots \arrow[d]                            &                                  & \vdots \arrow[d]                 &                       \\
                                 & BSO(2n) \arrow[rr]                          &                                  & BSO(2n)                          &                      
\end{tikzcd}$$

In the diagram $$\begin{tikzcd}
F^1 \arrow[d, "g^1_{4k+2}"] \arrow[r, "r"] & F^3 \arrow[d, "g^3_{4k+2}"] \\
F_{4k+2}^1 \arrow[r]         & F_{4k+2}^3   
\end{tikzcd}$$ the vertical arrows induce isomorphisms on $\pi_{\leq 4k+2}$ and $r$ induces the map $\mathbb{Z} \xrightarrow{\times \ell} \mathbb{Z}$ on $\pi_{4k+2}$ by \Cref{lemma}(c). Therefore, the map $F_{4k+2}^1 \to F_{4k+2}^3$ induces $\mathbb{Z} \xrightarrow{\times \ell} \mathbb{Z}$ on $\pi_{4k+2}$, up to sign.  

\begin{equation}\label{big}\begin{tikzcd}
{K(\mathbb{Z}, 4k+2)} \arrow[rr] \arrow[dd] \arrow[rd]                                                  &                                                                       & F_{4k+2}^1 \arrow[dd] \arrow[rd] &                       \\
                                                                                                        & {K(\mathbb{Z}, 4k+2)} \arrow[rr] \arrow[dd]                           &                                  & F_{4k+2}^3 \arrow[dd] \\
E_{4k+2}^1 \arrow[dd] \arrow[rr, "\ \ \ \ \ \ \ \ \ \ \ \ \ \ \ \ \ \ \ \ \ \ \ \ \ \ \ id"] \arrow[rd] &                                                                       & E_{4k+2}^1 \arrow[dd] \arrow[rd] &                       \\
                                                                                                        & E_{4k+2}^3 \arrow[rr, "\!\!\!\!\!\!\!\!\!\!\!\!\!\!\! id"] \arrow[dd] &                                  & E_{4k+2}^3 \arrow[dd] \\
E_{4k+1}^1 \arrow[rr] \arrow[rd]                                                                        &                                                                       & BSO(2n) \arrow[rd, "id"]         &                       \\
                                                                                                        & E_{4k+1}^3 \arrow[rr]                                                 &                                  & BSO(2n)              
\end{tikzcd}\end{equation}

Since the fibers $F_{4k+1}^1$ and $F_{4k+1}^3$ of the maps $E_{4k+1}^1 \to BSO(2n)$ and $E_{4k+1}^3 \to BSO(2n)$ have $\pi_{> 4k+1} = 0$, we see that the morphisms $E_{4k+1}^i \to BSO(2n)$, for $i = 1,3$ induce injections on $\pi_{4k+2}$ and isomorphisms on $\pi_{\geq 4k+3}$. Now consider the following part of the map between long exact sequences in homotopy induced either by the back face (for $i=1$) or the front face (for $i=3$) of the above diagram:

$$\begin{tikzcd}
\pi_{4k+3}E_{4k+2}^i \arrow[rr, "\sim"] \arrow[d]    &  & \pi_{4k+3}E_{4k+2}^i \arrow[d] \\
\pi_{4k+3}E_{4k+1}^i \arrow[rr, "\sim"] \arrow[d]    &  & \pi_{4k+3}BSO(2n) \arrow[d]    \\
{\pi_{4k+2}K(\mathbb{Z}, 4k+2)} \arrow[rr] \arrow[d] &  & \pi_{4k+2}F_{4k+2}^i \arrow[d] \\
\pi_{4k+2}E_{4k+2}^i \arrow[rr, "\sim"] \arrow[d]    &  & \pi_{4k+2}E_{4k+2}^i \arrow[d] \\
\pi_{4k+2}E_{4k+1}^i \arrow[rr, hook]                &  & \pi_{4k+2}BSO(2n)             
\end{tikzcd}$$

From the five lemma we conclude that the induced maps $\pi_{4k+2}K(\ZZ, 4k+2) \to \pi_{4k+2}F_{4k+2}^i$, for $i=1,3$, are isomorphisms $\ZZ \to \ZZ$. From the top face of diagram \eqref{big}, we conclude that the induced map $\pi_{4k+2}K(\ZZ, 4k+2) \to \pi_{4k+2}K(\ZZ, 4k+2)$ is the map $\ZZ \xrightarrow{\times \ell} \ZZ$ given by multiplication by $\ell$, up to sign. 

Now, recall that the fiber of the map $BSO(4k+2) \to BSO(2n)$, namely the Stiefel manifold $SO(2n)/SO(4k+2)$, is $(4k+1)$--connected, and $\pi_{4k+2} SO(2n)/SO(4k+2) \cong \mathbb{Z}$. Therefore $E_{4k+1}^3 \to BSO(2n)$ is a homotopy equivalence, and since $BSO(4k+2) \to E_{4k+2}^3$ induces an injection on $H^{4k+3}(-;\ZZ)$, we see that the $k$-invariant $E_{4k+1}^3 \to K(\ZZ, 4k+3)$ is the integral Stiefel--Whitney class $W_{4k+3}$. Recall, from the Serre long exact sequence for the fibration $$E^3_{4k+2} \xrightarrow{e_{4k+2}} E^3_{4k+1} \xrightarrow{\kappa} K(\ZZ, 4k+3)$$ we have that the kernel of $e_{4k+2}^*$ in degree $4k+3$ consists of multiples of the $k$-invariant $\kappa$. Therefore, $W_{4k+3} = m\kappa$ for some integer $m$. Since $W_{4k+3}$ has non-zero mod 2 reduction, namely $w_{4k+3}$, we see that $m$ is odd. Now note that $m(\kappa - W_{4k+3}) = 0$ since $W_{4k+3}$ is two-torsion. As the integral cohomology of $BSO(2n)$ contains no odd torsion \cite{B82}, we see that $\kappa = W_{4k+3}$.

Consider now at last the map of fibrations $$\begin{tikzcd}
{K(\mathbb{Z}, 4k+2)} \arrow[d] \arrow[r] & {K(\mathbb{Z}, 4k+2)} \arrow[d] \\
E_{4k+2}^1 \arrow[d] \arrow[r]            & E_{4k+2}^3 \arrow[d]            \\
E_{4k+1}^1 \arrow[r, "\varphi_{4k+1}"]                      & E_{4k+1}^3                     
\end{tikzcd}$$ 

where $\varphi_{4k+1}$ is the induced map between the stages $E_{4k+1}^1$ and $E_{4k+1}^3$ of the Moore--Postnikov systems. The $k$--invariants are given by the transgression $\tau$ of the corresponding fiber generator $\iota \in H^{4k+2}(K(\ZZ, 4k+2), 4k+2)$ in the Serre spectral sequence. Since the map $K(\ZZ, 4k+2) \to K(\ZZ, 4k+2)$ on fibers induces multiplication by $\pm\ell$ on $H^{4k+2}(-;\ZZ)$, by naturality of the transgression we have $$\varphi_{4k+1}^*W_{4k+3} = \varphi_{4k+1}^*(\tau(\iota)) = \tau(\pm \ell \iota) = \pm \ell \tau(\iota) = \pm \ell \mathfrak{o}.$$ Recall that $\ell = (2k)!$ if $k$ is even, and $\ell = \tfrac{1}{2}(2k)!$ if $k$ is odd. Since the obstruction $\mathfrak{o}$ is only well-defined up to sign, due to the non-trivial automorphism of $K(\ZZ, 4k+3)$ corresponding to multiplication by $-1$, this proves \Cref{TheoremI}.

\section{Proof of Massey's Theorem II}\label{TheoremIIsection}

Suppose $\xi \to X$ is an $SO(4k)$-bundle with a lift to $BU(2k)$ over the $(4k-1)$--skeleton of $X$. The obstruction to extending the section over the $4k$--skeleton is the pullback of the $k$-invariant $E_{4k-2} \to K(\pi_{4k-1} SO(4k)/U(2k), 4k)$ in the Moore--Postnikov system of $BU(2k) \to BSO(4k)$. If $k$ is odd, then $\pi_{4k-1} SO(4k)/U(2k) \cong \ZZ$, and if $k$ is even, $\pi_{4k-1} SO(4k)/U(2k) \cong \ZZ \oplus \ZZ_2$.

As a warm up, we recover the following refinement of Wu's theorem \cite[Théorème IV.10]{W52}, proved also in \cite[4.6]{HiHo58}, which is the $k = 1$ case of \cite[Theorem II]{M61}:

\begin{thm}\label{warmupThmII} Suppose $\xi \to X$ is an $SO(4)$ bundle over a CW complex $X$, and that we have reduced the structure group over the three--skeleton of $X$ to $U(2)$. The second obstruction $\mathfrak{o}$ to reducing the structure group to $U(2)$ satisfies $$p_1 - c_1^2 + 2e = 4\mathfrak{o}.$$ 
\end{thm}

\begin{proof}
Consider the beginning of the Moore--Postnikov system for the map $BU(2) \to BSO(4)$, 

$$\begin{tikzcd}
                             & BU(2) \arrow[d]                &                    \\
                             & \vdots \arrow[d]               &                    \\
                             & E_3 \arrow[d]                  &                    \\
{K(\mathbb{Z}, 2)} \arrow[r] & E_2 \arrow[d] \arrow[r, "\kappa"] & {K(\mathbb{Z}, 4)} \\
                             & BSO(4) \arrow[r, "W_3"]               & {K(\mathbb{Z}, 3)}
\end{tikzcd}$$

We compute $H^4(E_2;\mathbb{Z})$ using the Serre spectral sequence. (Recall $E_2 = BSpin^c(4)$; for a presentation of the cohomology of $BSpin^c(n)$ for all $n$, we refer the reader to \cite{D18}.) The third page of the spectral sequence in integral cohomology associated to $K(\ZZ, 2) \to E_2 \to BSO(4)$ is given by

$$
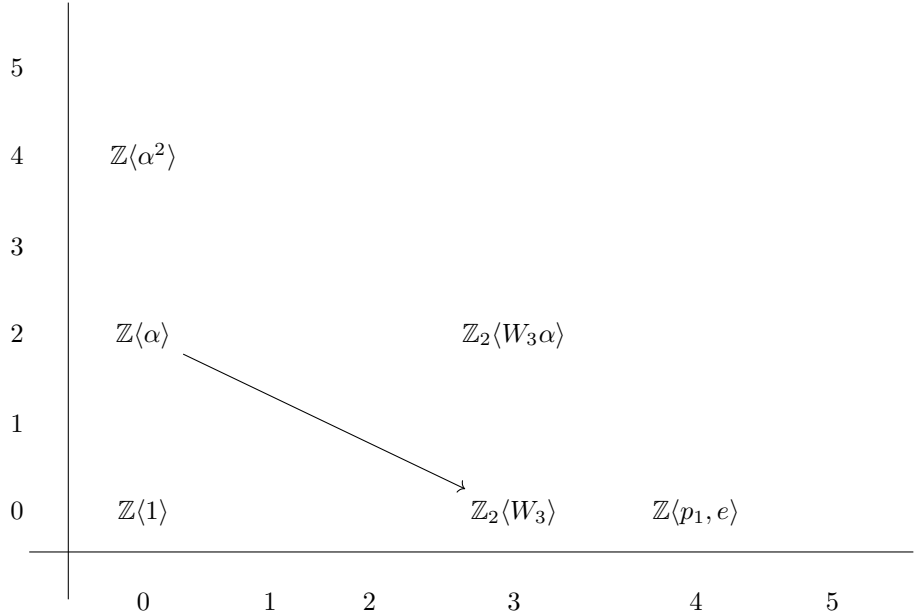
\begin{tikzcd}
{}                                                                                &                                                       &   &   &                                               &                                           &    \\
4                                                                                 & \large \mathbb{Z}\langle \alpha^2 \rangle             &   &   &                                               &                                           &    \\
3                                                                                 &                                                       &   &   &                                               &                                           &    \\
2                                                                                 & \large \mathbb{Z}\langle \alpha \rangle \arrow[rrrdd] &   &   & \large \mathbb{Z}_2\langle W_3 \alpha \rangle &                                           &    \\
1                                                                                 &                                                       &   &   &                                               &                                           &    \\
0                                                                                 & \large \mathbb{Z}\langle 1 \rangle                    &   &   & \large \mathbb{Z}_2\langle W_3 \rangle        & {\large \mathbb{Z}\langle p_1, e \rangle} &    \\
{} \arrow[rrrrrr, no head, shift left=8, shorten >=25pt] \arrow[uuuuuu, no head, shift right=8] & 0                                                     & 1 & 2 & 3                                             & 4                                         & {}
\end{tikzcd}$$

We have $d_3(\alpha) = W_3$, and so $d_3(2\alpha) = 0 $ and $d_3(\alpha^2) = 2\alpha W_3 = 0$. 

Denote the inclusion $K(\ZZ, 2) \to E_2$ by $i$. We conclude now from the spectral sequence that there exists a unique class $\gamma \in H^2(E_2; \ZZ)$ such that $i^*\gamma = 2\alpha$, and that there exists a class $\beta \in H^4(E_2; \ZZ)$ such that $i^*\beta = \alpha^2$ and $\{p_1, e, \beta \}$ is a $\ZZ$-basis for $H^4(E_2;\ZZ) \cong \ZZ^3$. Now, $\gamma^2$ is a linear combination of $p_1$, $e$, and $\beta$, but since $i^*\gamma^2 = 4\alpha^2 = 4i^*\beta$ and $i^*p_1 = i^*e = 0$, we see that $$\gamma^2 = ap_1 + be + 4\beta$$ for some integers $a,b$.

Since $BU(2) \xrightarrow{h_2} E_2$ induces an isomorphism on $H^2(-;\ZZ)$ and $H^2(-;\ZZ_2)$, we have $h_2^* \gamma = \pm c_1$ and hence $\rho_2(\gamma) = w_2$. The Pontryagin square formula \eqref{psquare} for $m=1$ reads $$\mathfrak{P}(w_2) = \rho_4(p_1) + \theta_2(w_4).$$ By \Cref{pontryaginsquarelemma}, we have $\mathfrak{P}(w_2) = \rho_4(\gamma^2)$, and by equation \eqref{2x-identity}, we have $\theta_2(w_4) = \rho_4(2e)$ since $\rho_2(e) = w_4$. Thus, we have $$\rho_4(\gamma^2) = \rho_4(p_1) + \rho_4(2e).$$ Therefore $\gamma^2 - p_1 - 2e = (a-1)p_1 + (b-2)e + 4\beta$ is divisible by four, so we have that $a-1$ and $b-2$ are divisible by four. By a change of basis preserving $p_1$ and $e$ and relabelling $\beta + \tfrac{a-1}{4} p_1 + \tfrac{b-2}{4} e$ to $\beta$, we have $$\gamma^2 = p_1 + 2e + 4\beta.$$

Now, the $k$-invariant $\kappa$ equals $kp_1 + \ell e + m \beta$, for some integers $k,\ell, m$. Pulling this class back by the map $BU(2) \xrightarrow{h_2} E_2$ must yield zero. Now, we have $4\kappa = 4kp_1 + 4\ell e + m(\gamma^2 - p_1 - 2e)$, and so  $$0 = 4h_2^*\kappa = 4k(c_1^2-2c_2) + 4\ell c_2 + m(c_1^2 - (c_1^2-2c_2) - 2c_2) = 4kc_1^2 + (4\ell - 8k)c_2.$$ Therefore, since $\{c_1^2, c_2\}$ is a $\ZZ$-basis for $H^4(BU(2);\ZZ) \cong \ZZ^2$, it follows that $k = \ell = 0$. We conclude $\kappa = m\beta$, and so $m\beta$ generates the kernel of the map $H^4(E_2;\ZZ) \xrightarrow{e_3^*} H^4(E_3;\ZZ)$. The map $BU(2) \to E_3$ induces an injection on $H^4(-;\ZZ)$. If $m = 0$, then $H^4(E_3; \mathbb{Z})$ would be isomorphic to $\mathbb{Z}^3$, which cannot inject into $H^4(BU(2);\mathbb{Z}) \cong \mathbb{Z}^2$. If $|m| > 1$, $\beta$ would pull back to be a torsion class in the torsion-free group $H^4(BU(2);\ZZ)$. Therefore $\kappa = \pm \beta$. That is, the $k$-invariant $\kappa$ satisfies (up to sign, as it is only well-defined up to sign) $$4\kappa = \gamma^2 - p_1 - 2e.$$ Pulled back to $X$, this becomes $4\mathfrak{o} = c_1^2 - p_1 - 2e$, as we have a lift all the way to $BU(2)$ over the three--skeleton of $X$, and $h_2^*\gamma = \pm c_1$. \end{proof}

\begin{ex} Consider the complex projective plane $\overline{\mathbb{CP}^2}$ with its opposite orientation. Choose a generator $\alpha$ for $H^2(\overline{\mathbb{CP}^2};\ZZ) \cong \ZZ$. Then $-\alpha^2$ is the oriented generator for $H^4(\overline{\mathbb{CP}^2};\ZZ) \cong \ZZ$, i.e. $\langle -\alpha^2, [\overline{\mathbb{CP}^2}]\rangle = 1$. Suppose we have a reduction of structure group for the tangent bundle to $U(2)$ over the three--skeleton. We have $p_1 = 3\alpha^2$; furthermore, $w_2(\overline{\mathbb{CP}^2}) = \rho_2(\alpha)$, so $c_1 = (2m+1)\alpha$ for some integer $m$. The Euler class is given by $e = -3\alpha^2$. Then $$2e - c_1^2 + p_1 = (-6-(2m+1)^2 + 3)\alpha^2,$$ which evaluates to $4(m^2+m+1)$ when paired with the fundamental class $[\overline{\mathbb{CP}^2}]$. This is divisible by four, but non-zero for all $m$. The obstruction $\mathfrak{o}$ is, up to sign, $(m^2 + m + 1)\alpha^2$.

By contrast, suppose we have a reduction to $U(2)$ of the structure group of the tangent bundle to $\mathbb{CP}^2$ over its three--skeleton. Then $c_1 = (2m+1)\alpha$ as before, but $\langle \alpha^2, [\mathbb{CP}^2] \rangle = 1$ and $e = 3\alpha^2$. We have $$\langle 2e - c_1^2 + p_1, [\mathbb{CP}^2] \rangle = 4(2 - m^2 - m).$$ Therefore the obstruction $\mathfrak{o}$ vanishes for $m = -2$ and $m = 1$, i.e. $c_1 = \pm 3 \alpha$. \end{ex}

Generally, for any $k$, consider the class $\sum_{i=0}^{2k} (-1)^i c_ic_{2k-i} - (-1)^k p_k$. Here, $c_i$, for $0\leq i \leq 2k-1$, denotes the Chern classes of the $U(n)$ bundle induced over the $(4k-1)$--skeleton of $X$, while $c_{2k}$ denotes the Euler class of $\xi$. 

\begin{thm}\cite[Theorem II]{M61}\label{theoremII} Suppose $n = 2k$. Let $\theta \to X$ be the $SO(2n)/U(n)$ bundle associated to an oriented rank $2n$ real vector bundle $\xi$ over $X$, and let $s$ be a section of the bundle over the $(2n-1)$--skeleton of $X$. The obstruction $\mathfrak{o}_{2n} = \mathfrak{o}_{4k}$ is an integral class if $k$ is odd, while if $k$ is even, it is a pair consisting of an integral class and a mod 2 class. If $k$ is odd, then $$\sum_{i+j = 2k} (-1)^i c_ic_j - (-1)^k p_k = 4 \mathfrak{o}_{2n},$$ while for $k$ even, this same formula holds with $\mathfrak{o}_{2n}$ replaced by its integral component\footnote{The integral component is only well-defined once a choice of isomorphism between $\pi_{2n-1}SO(2n)/U(n)$ and $\ZZ \oplus \ZZ_2$ is made. However, up to sign as usual, $4\mathfrak{o}_{2n}$ does not depend on this choice.}. In the formula, $c_0, \ldots, c_{n-1}$ are the Chern classes of the induced $U(n)$ bundle over the $(2n-1)$--skeleton of $X$, while $c_n$ denotes the Euler class of $\xi$.\end{thm}


\begin{proof} Recall equation \eqref{psquare} involving the Pontryagin square:

$$\mathfrak{P}(w_{2k}) = \rho_4(p_k) + \theta_2\left( \sum_{i=0}^{k-1}w_{2i}w_{4k-2i} \right).$$

As argued in \Cref{integralSW}, we have that $\theta_2( \sum_{i=0}^{k-1}w_{2i}w_{4k-2i})$ is the mod 4 reduction of $2\sum_{i=0}^{k-1} c_i c_{2k-i}$, or equivalently, the mod 4 reduction of $\sum_{i=0}^{2k}(-1)^i c_i c_{2k-i} - (-1)^k c_k^2$. Since $\mathfrak{P}(w_{2k}) = \rho_4(c_k^2)$ by \Cref{pontryaginsquarelemma}, we conclude that \begin{align*} \rho_4\left( \sum_{i=0}^{2k} (-1)^i c_ic_{2k-i} - (-1)^k p_k \right) &= \rho_4\left( \sum_{i=0}^{2k} (-1)^i c_ic_{2k-i} - (-1)^k c_k^2 \right) + \rho_4\left( (-1)^kc_k^2 - (-1)^k p_k \right) \\ &= \theta_2\left( \sum_{i=0}^{k-1}w_{2i}w_{4k-2i}\right) + (-1)^k \mathfrak{P}(w_{2k}) - (-1)^k\rho_4(p_k)\\ &= (-1)^k\left[\mathfrak{P}(w_{2k}) - \rho_4(p_k) + (-1)^k\theta_2\left( \sum_{i=0}^{k-1}w_{2i}w_{4k-2i}\right)\right].\end{align*} Since $\theta_2$ is the map on cohomology induced by the map $\ZZ_2 \to \ZZ_4$ sending $[1]$ to $[2]$, we have $\theta_2 = -\theta_2$. Therefore the righthand side in the above equation vanishes regardless of the parity of $k$.

Now, the $k$-invariant $E_{4k-2} \xrightarrow{\kappa} K(\pi_{4k-1} SO(4k)/U(2k), 4k)$ is either an integral cohomology class, or a pair consisting of an integral cohomology class and a mod 2 cohomology class. In either case, let $\mathfrak{o}$ denote the integral component. Consider the Serre long exact sequence in integral cohomology for the fibration $E_{4k-1} \xrightarrow{e_{4k-2}} E_{4k-2} \xrightarrow{\kappa} K(\pi_{4k-1} SO(4k)/U(2k), 4k)$, namely, $$ H^{4k}(K(\pi_{4k-1} SO(4k)/U(2k), 4k);\ZZ) \xrightarrow{\kappa^*} H^{4k}(E_{4k-2}; \ZZ) \xrightarrow{e_{4k-2}^*} H^{4k}(E_{4k-1}; \ZZ).$$ Note that $H^{4k}(K(\ZZ_2, 4k); \ZZ) = 0$, and hence regardless of the parity of $k$, we have $$H^{4k}(K(\pi_{4k-1} SO(4k)/U(2k), 4k);\ZZ) \cong H^{4k}(K(\ZZ, 4k);\ZZ) \cong \ZZ ,$$ generated by the class $\iota_{4k}$ generating $H^{4k}(K(\ZZ, 4k);\ZZ)$. From the long exact sequence we thus see that the kernel of $e_{4k-2}^*$ on $H^{4k}(-;\ZZ)$ is the subgroup of $H^{4k}(E_{4k-2}; \ZZ)$ generated by $\kappa^*\iota_{4k}$, which we denote by $\mathfrak{o}$ (this class pulls back to $\mathfrak{o}_{4k}$ on $X$).

Since the map $BU(2k) \to E_{4k-1}$ induces an injection on $H^{4k}(-;\ZZ)$, and the class $\sum_{i=0}^{2k} (-1)^i c_ic_{2k-i} - (-1)^k p_k \in H^{4k}(E_{4k-2}; \ZZ)$ pulls back to be 0 in $H^{4k}(BU(2k);\ZZ)$, we have $$\sum_{i=0}^{2k} (-1)^i c_ic_{2k-i} - (-1)^k p_k \in \ker(e_{4k-2}^*),$$ and hence there is some integer $\ell$ such that $$\sum_{i=0}^{2k} (-1)^i c_ic_{2k-i} - (-1)^k p_k = \ell \mathfrak{o}.$$

Before proceeding, we show that $\mathfrak{o}$ generates a $\ZZ$ summand in $H^{4k}(E_{4k-2};\ZZ)$. First, if there were a non-zero integer $m$ such that $m\mathfrak{o} = 0$, then we would have $m (\sum_{i=0}^{2k} (-1)^i c_ic_{2k-i} - (-1)^k p_k) = 0 \in H^{4k}(E_{4k-2}; \ZZ)$. Now consider the classifying map for the tangent bundle of the sphere $S^{4k} \to BSO(4k)$. Since $S^{4k}$ can be built with one 0--cell and one $4k$--cell, this map lifts to $E_{4k-2}$. Pulling back  $\sum_{i=0}^{2k} (-1)^i c_ic_{2k-i} - (-1)^k p_k$ to $S^{4k}$ yields twice the Euler class ($p_k$ is trivial as the sphere is stably parallelizable), i.e. the element $4 \in H^{4k}(S^{4k};\ZZ) \cong \ZZ$. If $m((\sum_{i=0}^{2k} (-1)^i c_ic_{2k-i} - (-1)^k p_k) = 0$, then $4m = 0$, a contradiction. Secondly, we observe that $\mathfrak{o}$ is a primitive class. Suppose that $\mathfrak{o} = m\alpha$, for some $\alpha \in H^{4k}(E_{4k-2}; \ZZ)$. Then, as $\mathfrak{o}$ is non-zero by the previous argument, $\alpha$ would pull back to be a non-zero class of order $m$ in $H^{4k}(BU(2k);\ZZ)$. Since the integral cohomology of $BU(2k)$ is torsion-free, we conclude that $m = \pm 1$.

Since the mod 4 reduction of $\sum_{i=0}^{2k} (-1)^i c_ic_{2k-i} - (-1)^k p_k$ is zero, we conclude that $$\sum_{i=0}^{2k} (-1)^i c_ic_{2k-i} - (-1)^k p_k = 4m\mathfrak{o}$$ for some integer $m$. Pulling back again to the sphere $S^{4k}$ we see that $4m\mathfrak{o}_{4k} = 4 \in H^{4k}(S^{4k};\ZZ)$, so $m = \pm 1$ (as $\mathfrak{o}_{4k}$ is only defined up to sign). \end{proof}


\begin{ex} Consider the quaternionic projective plane $\mathbb{HP}^2$, and suppose we have a reduction of structure group to $U(4)$ of its tangent bundle over its seven--skeleton. To see that such a reduction indeed exists, first note that the homotopy groups of $SO(8)/U(4)$ in degrees one through six are isomorphic to those of $SO/U$, which are $0, \ZZ, 0, 0, 0, \ZZ$. By \Cref{TheoremI}, the two obstructions we encounter are $W_3$ and $W_7$, which vanish for degree reasons. 

Now, \Cref{theoremII} tells us that the obstruction to extending this over the eight--skeleton satisfies $2e + c_2^2 - p_2 = 4\mathfrak{o_8}$. First of all, the Pontryagin classes of $\mathbb{HP}^2$ are obtained from the formula $1 + p_1 + p_2 = \frac{(1+u)^6}{1+4u}$, where $u$ denotes a particular choice of generator for $H^4(\mathbb{HP}^2;\ZZ)$; see \cite[\textsection 2]{M62}. This yields $p_1 = 2u$ and $p_2 = 7u^2$. From the general identity $p_1 = c_1^2 - 2c_2$, since $c_1$ necessarily vanishes, we conclude $c_2 = -u$. We therefore have $$\langle 2e + c_2^2 - p_2, [\mathbb{HP}^2] \rangle = 0,$$ and since $H^8(\mathbb{HP}^2;\ZZ)$ is torsion-free, we conclude $\mathfrak{o}_8 = 0$. This same calculation holds for $\overline{\mathbb{HP}^2}$. We emphasize that $\mathbb{HP}^2$ and $\overline{\mathbb{HP}^2}$ do not admit almost complex structures, as proved in \cite{M62}. In fact, Massey shows that $\mathbb{HP}^n$ does not even admit a stable almost complex structure, for all $n \geq 2$. The point here is that the integral part of the obstruction to extending the almost complex structure over the eight--skeleton of $\mathbb{HP}^2$ vanishes, but the $\ZZ_2$ part does not. \Cref{theoremII} does not detect the latter. For some further discussion on this, see Section 6 below. \end{ex}

We now address the second obstruction to reducing the structure group of an $SO(6)$ bundle to $U(3)$. 

\begin{proposition}\label{secondobstruction2n=6} The second obstruction $\mathfrak{o}$ to reducing the structure group of an $SO(6)$ bundle to $U(3)$ satisfies $$-2c_1c_3 + c_2^2 - p_2 = 4\mathfrak{o}.$$ \end{proposition}

\begin{proof} First of all, from the well known identification of $SO(6)/U(3)$ with $\mathbb{CP}^3$, we see that after $\pi_2$, the first non-zero homotopy group is $\pi_7 SO(6)/U(3) \cong \mathbb{Z}$. 

Next, consider the fiber bundle $SO(6)/U(3) \to SO(8)/U(4) \to S^6$. Since $S^6$ admits an almost complex structure, this bundle has a section, and hence the long exact sequence in homotopy groups splits. Therefore, $\pi_7 SO(6)/U(3) \cong \ZZ$ maps isomorphically onto the $\ZZ$ summand of $\pi_7 SO(8)/U(4) \cong \ZZ \oplus \ZZ_2$. 

Now we look at the map of Moore--Postnikov systems induced by the map of fibrations $$\begin{tikzcd}
SO(6)/U(3) \arrow[d] \arrow[r] & SO(8)/U(4) \arrow[d] \\
BU(3) \arrow[d] \arrow[r]      & BU(4) \arrow[d]      \\
BSO(6) \arrow[r]               & BSO(8)              
\end{tikzcd}$$

The relevant part for us will be the following map of fibrations:

$$\begin{tikzcd}
{K(\mathbb{Z}, 7)} \arrow[d] \arrow[r] & {K(\mathbb{Z}\oplus \mathbb{Z}_2, 7)} \arrow[d] \\
E_7^1 \arrow[d] \arrow[r]                & E_7^2 \arrow[d]                                  \\
E_6^1 \arrow[r]                          & E_6^2                                           
\end{tikzcd}$$

Arguing as in the last part of \Cref{theoremIsection}, involving diagram \eqref{r}, we see that the map $K(\ZZ, 7) \to K(\ZZ \oplus \ZZ_2, 7)$ induces the inclusion of the $\ZZ$ summand on $\pi_7$, as $SO(6)/U(3) \to SO(8)/U(4)$ does. Therefore, by naturality of the transgression for the Serre spectral sequence with $\ZZ$ coefficients again, and since $H^8(K(\ZZ, 8);\ZZ) \cong H^8(K(\ZZ\oplus \ZZ_2, 8); \ZZ)$ induced by the inclusion $\ZZ \hookrightarrow \ZZ \oplus \ZZ_2$ of the left summand, the $k$-invariant $E_6 \to K(\ZZ,8)$ is the pullback of the class $\mathfrak{o}_8$ in \Cref{theoremII}. In particular, four times the $k$-invariant is the pullback of $2e - 2c_1c_3 + c_2^2 - p_2$. Since the Euler class $e$ of $BSO(8)$ pulls back to zero in $BSO(6)$, we have that four times the $k$-invariant (i.e. four times the second obstruction $\mathfrak{o}$) is $-2c_1c_3 + c_2^2 - p_2$. \end{proof}

\begin{ex}
There are examples of $SO(6)$ bundles over eight--manifolds for which the first obstruction to reducing the structure group to $U(3)$ vanishes, but the second obstruction does not -- note, over an eight-dimensional base, the second obstruction is also the last obstruction. For example, by \cite[Lemma B.1]{BK03}, there exists an orientable rank 6 bundle $E \to S^8$ such that $p_2 \neq 0$. Note that the first obstruction, namely $W_3$, vanishes as $H^3(S^8; \mathbb{Z}) = 0$. Due to the cohomology of $S^8$, the second obstruction $\mathfrak{o}$ satisfies $4\mathfrak{o} = -p_2 \neq 0$, so $\mathfrak{o} \neq 0$.
\end{ex}

\section{Remarks on Massey's Theorem III}

In \cite[Theorem III]{M61}, the degree eight obstruction $\mathfrak{o}_8 \in H^8(-;\mathbb{Z}_2)$ to a stable almost complex structure is discussed. Concretely, suppose we have an $SO(2n)$-bundle over a CW complex $X$, such that $n\geq 5$, with a section $s$ of the associated $SO(2n)/U(n)$-bundle over the seven--skeleton of $X$. Denoting the total space of the $SO(2n)/U(n)$-bundle by $E$, there are classes $u \in H^2(E;\mathbb{Z})$ and $v \in H^6(E;\mathbb{Z})$ such that $s^*u = 0$, $s^*v = 0$, and $u,v$ restrict to generators $x,y$ of the degree $\leq 8$ cohomology of $SO(2n)/U(n)$.  The obstruction $\mathfrak{o}_8$ is then claimed to be a certain class $b_8$ pulled back from $X$ that appears in the expansion of $Sq^2\rho_2(v)$ in terms of the cohomology of $X$ and $SO(2n)/U(n)$. Namely, 

$$Sq^2\rho_2(v) = p^*b_8 + (p^*b_6)\rho_2(u) + (p^*b_4)\rho_2(u)^2 + (p^*b_2)\rho_2(u)^3 + (p^*b_2')\rho_2(v) + \rho_2(u)^4.$$

Such an equation holds since $x$ and $y$ transgress to $W_3$ and $W_7$, which are assumed to vanish, and hence the spectral sequence collapses in sufficiently low degree.

We have the relation $Sq^2\rho_2(y) = \rho_2(x)^4$, which can be seen by combining \cite[Theorem 6.11]{MT} and \cite[Theorem 6.6]{MT}. It is clear that the vanishing of the obstruction $\mathfrak{o}_8$ implies the vanishing of $b_8$. Indeed, if the obstruction vanishes, then one can extend $s$ over the eight--skeleton of $X$, in which case we can apply the pull back by this extended section to the above equation, obtaining $b_8 = 0$ since $s^*u = s^*v = 0$.

In the cases of most interest, namely stable almost complex structures on closed eight--manifolds, this obstruction was addressed by Heaps \cite{H60} and Thomas \cite{T67}. In \cite{H60}, the obstruction itself is not explicitly identified, but it is shown (making use of another theorem of Massey that $W_7=0$ for the tangent bundle of a closed orientable eight--manifold) that a closed orientable eight--manifold admits a stable almost complex structure if and only if $W_3 = 0$ and $w_8$ is in the image of $Sq^2\rho_2$ \cite[Corollary 1.3]{H60}. Thomas \cite[Theorem 1.2]{T67} showed that a general real orientable vector bundle $\xi$ over a CW complex admits a stable almost complex structure if and only if $W_3(\xi) = 0, W_7(\xi) = 0$, and a certain secondary cohomology operation $\Omega$ vanishes on $w_8(\xi) + w_4(\xi)^2 + w_2(\xi)^2w_4(\xi)$.


\begin{thebibliography}{9999999}
\bibitem[Ba06]{Ba} Baues, H.J., 2006. \emph{Obstruction theory: On homotopy classification of maps} (Vol. 628). Springer.

\bibitem[BK03]{BK03} Belegradek, I. and Kapovitch, V., 2003. \emph{Obstructions to nonnegative curvature and rational homotopy theory}. Journal of the American Mathematical Society, 16(2), pp.259-284.

\bibitem[BS53]{BS53} Borel, A. and Serre, J.P., 1953. \emph{Groupes de Lie et puissances r\'eduites de Steenrod}. American Journal of Mathematics, pp.409-448.

\bibitem[BT62]{BT62} Browder, W. and Thomas, E., 1962. \emph{Axioms for the generalized Pontryagin cohomology operations}. The Quarterly Journal of Mathematics, 13(1), pp.55-60.

\bibitem[B82]{B82} Brown Jr, E.H., 1982. \emph{The cohomology of $BSO(n)$ and $BO(n)$ with integer coefficients}. Proceedings of the American Mathematical Society, Vol. 85, No. 2, pp. 283-288.

\bibitem[D18]{D18} Duan, H., 2018. \emph{The characteristic classes and Weyl invariants of Spinor groups}. arXiv preprint arXiv:1810.03799.

\bibitem[E50]{E50} Ehresmann, C., 1950, September. \emph{Sur les variétés presque complexes}. In the Proceedings of the International Congress of Mathematicians, Cambridge, Massachusetts (Vol. 2, pp. 412-419). 

\bibitem[E52]{E52} Ehresmann, C., 1952. \emph{Sur les var\'et\'es presque complexes}. In: Proceedings of the International Congress of Mathematicians, Cambridge, Mass., Aug. 30–Sept. 6, 1950. Vol. 2. pp. 412–419.

\bibitem[Ha63]{Ha63} Harris, B., 1963. \emph{Some calculations of homotopy groups of symmetric spaces}. Transactions of the American Mathematical Society, 106(1), pp.174-184.

\bibitem[H60]{H60} Heaps, T., 1970. \emph{Almost complex structures on eight-and ten-dimensional manifolds}. Topology, 9(2), pp.111-119.

\bibitem[He59]{He59} Hermann, R., 1959. \emph{Secondary obstructions for fibre spaces}. Bulletin of the American Mathematical Society, 65(1), pp.5-8.

\bibitem[He60]{He60} Hermann, R., 1960. \emph{Obstruction theory for fibre spaces}. Illinois Journal of Mathematics, 4(1), pp.9-27.

\bibitem[HiHo58]{HiHo58} Hirzebruch, F. and Hopf, H., 1958. \emph{Felder von Flächenelementen in 4-dimensionalen Mannigfaltigkeiten}. Mathematische Annalen, 136, pp.156-172.

\bibitem[Ka63]{Ka63} Kahn, D.W., 1963. \emph{Induced maps for Postnikov systems}. Transactions of the American Mathematical Society, 107(3), pp.432-450.

\bibitem[K59]{K59} Kervaire, M.A., 1959. \emph{A note on obstructions and characteristic classes}. American Journal of Mathematics, 81(3), pp.773-784.

\bibitem[M61]{M61} Massey, W.S., 1961. \emph{Obstructions to the existence of almost complex structures}. Bulletin of the American Mathematical Society, 67(6), pp.559-564.

\bibitem[M62]{M62} Massey, W.S., 1962. \emph{Non-existence of almost complex structures on quaternionic projective spaces}. Pacific Journal of Mathematics, 12, pp.1379–1384.

\bibitem[May67]{M} May, J.P., 1967. \emph{Simplicial objects in algebraic topology} (Vol. 11). University of Chicago Press.

\bibitem[MT91]{MT} Mimura, M. and Toda, H., 1991. \emph{Topology of Lie groups, I and II} (Vol. 91). American Mathematical Soc.

\bibitem[N10]{N10} Neisendorfer, J., 2010. \emph{Algebraic methods in unstable homotopy theory} (Vol. 12). Cambridge University Press.

\bibitem[St51]{St} Steenrod, N., 1951. \emph{The Topology of Fibre Bundles}, Volume 27. Princeton University Press.

\bibitem[S65]{S65} Sutherland, W.A., 1965. \emph{A note on almost complex and weakly complex structures}. Journal of the London Mathematical Society, 1(1), pp.705-712.

\bibitem[TV94]{TV} Teichner, P. and Vogt, E., 1994. \emph{All 4-manifolds have spin$^c$ structures}. Unpublished note, available from \texttt{https://people.mpim-bonn.mpg.de/teichner/Math/ewExternalFiles/spin.pdf}.

\bibitem[T60]{T60} Thomas, E., 1960. \emph{On the cohomology of the real Grassmann complexes and the characteristic classes of $n$-plane bundles}. Transactions of the American Mathematical Society, 96(1), pp.67-89.

\bibitem[T67]{T67} Thomas, E., 1967. \emph{Complex structures on real vector bundles}. American Journal of Mathematics, 89(4), pp.887-908.

\bibitem[Wh12]{Wh} Whitehead, G.W., 2012. \emph{Elements of homotopy theory} (Vol. 61). Springer Science \& Business Media.

\bibitem[W52]{W52} Wu, W.T., 1952. \emph{Sur les classes caract\'eristiques des structures fibr\'ees sph\'eriques}. Actualités Sci. Ind, 1183, pp.1-89.

\end{thebibliography}
\end{document}